\documentclass[a4paper]{amsart}

\usepackage{color}
\usepackage{graphicx}
\usepackage[normalem]{ulem}
\usepackage{morefloats}
\usepackage{hyperref}
\numberwithin{equation}{section}
\hypersetup{
    bookmarks=true,         
    unicode=false,          
    pdftoolbar=true,        
    pdfmenubar=true,        
    pdffitwindow=false,     
    pdfstartview={FitH},    
    pdftitle={My title},    
    pdfauthor={Author},     
    pdfsubject={Subject},   
    pdfcreator={Creator},   
    pdfproducer={Producer}, 
    pdfkeywords={keyword1, key2, key3}, 
    pdfnewwindow=true,      
    colorlinks=false,       
    linkcolor=green,          
    citecolor=green,        
    filecolor=green,      
    urlcolor=green,           
    urlbordercolor={1 1 1}  
}

\usepackage{amsmath, amsfonts, amsthm, amssymb}
\usepackage{fancyhdr}

\usepackage{biblatex}

\usepackage{mathrsfs}
\usepackage{blkarray}
\usepackage{amsthm}
\usepackage{mathtools}
\usepackage[title]{appendix}
\newtheorem{theorem}{Theorem}[section]
\newtheorem*{theorem*}{Theorem}
\newtheorem{lemma}[theorem]{Lemma}

\newcommand{\is}{\hspace{2pt}}

\newcommand{\dt}{\is dt}
\newcommand{\DD}{\mathbb{D}}

\newcommand{\Z}{\mathbb{Z}}

\newcommand{\C}{\mathbb{C}}

\newcommand{\N}{\mathbb{N}}
\newcommand{\RR}{\Rightarrow}

\newtheorem{proposition}[theorem]{Proposition}

\theoremstyle{definition}

\theoremstyle{remark}
\newtheorem{remark}{Remark}

\begin{document}

\title[\resizebox{2.9in}{!}{  \large A VON NEUMANN TYPE INEQUALITY FOR AN ANNULUS}]{\large A VON NEUMANN TYPE INEQUALITY FOR AN ANNULUS}

\author[\resizebox{1.15in}{!}{\small GEORGIOS TSIKALAS}]{\small GEORGIOS TSIKALAS}

\address{DEPARTMENT OF MATHEMATICS AND STATISTICS, WASHINGTON UNIVERSITY IN ST. LOUIS, ST. LOUIS, MO, 63136}
\email{gtsikalas@wustl.edu} 
\keywords{von Neumann inequality, complete Pick space, functional calculi, annulus. \\  \hspace*{0.34 cm} I thank Professor John E. M$^c$Carthy, my thesis advisor, for all the conversations about the topic and his helpful suggestions. I also thank the Onassis Foundation for providing financial support for my PhD studies.}
\small
\begin{abstract}
    \small
Let $A_r=\{r<|z|<1\}$ be an annulus. We consider the class of operators 
$$\mathcal{F}_r:=\{T\in\mathcal{B}(H): r^2T^{-1}(T^{-1})^*+TT^*\le r^2+1,\hspace{0.08 cm}\sigma(T)\subset A_r\}$$
and show that for every bounded holomorphic function $\phi$ on $A_r:$
 $$\sup_{T\in\mathcal{F}_r}||\phi(T)||\le\sqrt{2}||\phi||_{\infty},$$
 where the constant $\sqrt{2}$ is the best possible. We do this by characterizing the calcular norm induced on $H^{\infty}(A_r)$ by $\mathcal{F}_r$ as the multiplier norm of a suitable holomorphic function space on $A_r$.
\end{abstract}
\maketitle

 \section{INTRODUCTION}
 \large 
 Let $T$ be a contractive operator on a Hilbert space. A famous inequality due to von Neumann (\cite{vNeumann}) asserts that $$||p(T)||\le \sup_{\DD}|p|$$
 whenever $p$ is a polynomial in one variable. As an immediate consequence, we find that for every one-variable polynomial $p$:
 $$\sup_{||T||\le 1}||p(T)||=||p||_{\infty}.$$An extraordinary amount of work has been done by the operator theory community extending the inequality of von Neumann. We mention two well-known generalisations due to And\^{o} and Drury. And\^{o}'s inequality (\cite{Ando}) states that if $T=(T_1, T_2)$ is a contractive commuting pair of operators on a Hilbert space, then 
 $$||p(T)||\le \sup_{\DD^2}|p| $$
 whenever $p$ is a polynomial in two variables. Now, while the corresponding analogue of And\^{o}'s theorem and the von Neumann inequality fails for three or more contractions (see e.g. \cite{Varopoulos}), Drury's generalization (\cite{Drury}) does tell us that every row $d$-contraction $T$ on a Hilbert space (i.e. every $d$-tuple $T=(T_1,\dots,T_d)$ of commuting operators such that $T_1T_1^*+\dots+T_dT_d^*\le 1$) satisfies 
 $$||\phi(T)||\le ||\phi||_{\text{Mult}(H^2_d)}$$
 whenever $\phi$ is a polynomial in $d$ variables. Here, $||\phi||_{\text{Mult}(H^2_d)}$ denotes the norm of the multiplication operator $M_{\phi}f:=\phi f$ acting on the Drury-Arveson space $H^2_d$ over the (open) $d$-dimensional complex unit ball $\mathbb{B}_d$. Since the operators $M_{z_i}$ of multiplication by the independent variables form a row d-contraction $M_z:=(M_{z_1},\dots, M_{z_d})$ acting on $H^2_d$, Drury's inequality is actually equivalent to the statement 
 $$\sup_{T_1T_1^*+\dots+T_dT_d^*\le 1}||\phi(T_1,\dots,T_d)||=||\phi||_{\text{Mult}(H^2_d)},$$
 for every $d$-variable polynomial $\phi.$ \par
  \par In this note, we deduce a von Neumann type inequality for the class of operators $$\mathcal{F}_r:=\{T\in\mathcal{B}(H): r^2T^{-1}(T^{-1})^*+TT^*\le r^2+1,\hspace{0.08 cm}\sigma(T)\subset A_r\}$$
 associated with the annulus $A_r=\{r<|z|<1\}$ by applying standard positivity arguments to model formulas in an appropriate function space setting. This setting is introduced in Section \ref{spacesection}; it is the holomorphic function space $\mathscr{H}^2(A_r)$ induced on $A_r$ by the kernel 
 $$k_{A_r}(\lambda,\mu):=\frac{1-r^2}{(1-\lambda\bar{\mu})(1-r^2/\lambda\bar{\mu})},\hspace{0.3 cm} \forall \lambda,\mu\in A_r.$$
 Our main result is the following theorem, the proof of which is contained in Section \ref{proofsection}. 
  \begin{theorem}\label{1}
For every $\phi\in H^{\infty}(A_r)$, 
 $$\sup_{T\in\mathcal{F}_r}||\phi(T)||=||\phi||_{\textnormal{Mult}(\mathscr{H}^2(A_r))}\le\sqrt{2}||\phi||_{\infty},$$
 where the constant $\sqrt{2}$ is the best possible.
 \end{theorem}
\noindent The fact that $k_{A_r}$ is a complete Pick kernel (to be defined below) also allows us to show the following extension result. 
\begin{theorem} \label{9}
 Let $0<r<1.$ For every $\phi\in H^{\infty}(A_r)$, the quantity 
 $$ \resizebox{0.995\hsize}{!}{$\min\{
||\psi||_{\text{Mult}(H^2_2)}: \psi\in\textnormal{Mult}(H_2^2) \text{ and } \psi\Big(\frac{z}{\sqrt{r^2+1}},\frac{r}{\sqrt{r^2+1}}\frac{1}{z}\Big) =\phi(z)\text{,  } \forall z\in A_r\}$}$$
lies in the interval $[||\phi||_{\infty},\sqrt{2}||\phi||_{\infty}].$ Moreover, the constant $\sqrt{2}$ is the best possible.
\end{theorem}
 \noindent Here, $H^2_2$ denotes the $2$-dimensional Drury-Arveson space on the open unit ball $\mathbb{B}_2\subseteq\C^2$.

 \small
 \section{PRELIMINARIES}
 \large 
\par Let $X$ be a nonempty set. A function $k:X\times X\to \C$ is called positive semi-definite, if whenever $n\in \N$ and $x_1, \dots,x_n\in X$ and $w_1,\dots w_n\in \C,$ then $\sum_{i,j=1}^{n}k(x_j,x_i)w_i\overline{w_j}\ge 0.$ We also say that $k$ is a \textit{kernel}. For each $x\in X,$ define a function $k(\cdot,x)$ on $X$ by $k(\cdot,x)(y)=k(y,x).$ Define also an inner product on the linear span of these functions by 
 $$\big\langle\sum_i a_ik(\cdot,x_i),\sum_jb_jk(\cdot,x_j)\big\rangle=\sum_{i,j}a_i\overline{b_j}k(x_j,x_i).$$
 Let $\mathcal{H}_k$ denote the Hilbert space obtained by completing the linear span of the functions $k(\cdot,x)$ with respect to the previous inner product. We may regard vectors $f$ in $\mathcal{H}_k$ as functions on $X,$ with $f(x)=\langle f,k(\cdot,x)\rangle.$ \\
The \textit{multiplier algebra} $\text{Mult}(\mathcal{H}_k)$ is defined as the collection of functions $\phi:X\to\C$ such that $(M_{\phi}f)(x)=\phi(x)f(x)$ defines a bounded operator $M_{\phi}:\mathcal{H}_k\to\mathcal{H}_k.$ The multipliers $\phi$ with $||M_{\phi}||\le C$ are characterized by 
 \begin{equation} (C^2-\phi(y)\overline{\phi(x)})k(y, x)\ge 0, \label{2} \end{equation}
 since it is equivalent to $||M^*_{\phi}f||_{\mathcal{H}_k}\le C||f||_{\mathcal{H}_k},$ for a dense subset of $\mathcal{H}_k$. 
\\ 
Now, suppose $k$ is a kernel on $X.$ We say that $k$ is a (normalized) \textit{complete Pick kernel} (and the space $\mathcal{H}_k$ it induces we call a \textit{complete Pick space}) if there exists some (possibly infinite) cardinal $d$ and an injection $b:X\to\mathbb{B}_d$ ($\mathbb{B}_d$ being the (open) unit ball in some $d$- dimensional Hilbert space $K$) such that 
$$1-\frac{1}{k(y,x)}=\langle b(y),b(x)\rangle_K$$
for all $x$ and $y$ in $X.$ The standard definition is actually based on finite interpolation problems (its equivalence to the one we gave is a theorem of Agler and M$^c$Carthy \cite{CompleteNP}), however we will not be needing it for the purposes of this note. As a consequence of our definition, $\mathcal{H}_k$ can be identified with the $*$-invariant subspace 
$$\mathcal{H}_k=\text{ closed linear span of } \{k_x:x\in \text{ran }  b\}\subset H^2_d.$$
Here, $H^2_d$ denotes the $d$-dimensional Drury-Arveson  space, which is the reproducing kernel Hilbert space induced by the kernel 
$$a_d(\lambda,\mu)=\frac{1}{1-\langle\lambda,\mu\rangle_K}$$
and defined on $\mathbb{B}_d$. \\
The class of complete Pick kernels is well-known and extensively studied. A comprehensive treatment can be found in \cite{Pick}. In the sequel, we will only have occasion to use the following basic result. Suppose 
$$k(y,x)=\frac{1}{1-\langle b(y),b(x)\rangle}$$
is a complete Pick kernel on a set $X$ embedding into the Drury-Arveson space $H^2_d.$ Then, for every $\phi\in\text{Mult}(\mathcal{H}_k)$ we obtain that $||\phi||_{\text{Mult}(\mathcal{H}_k)}$ is equal to
\begin{equation}min\{||\psi||_{\textnormal{Mult}(H^2_d)}:\psi\in\text{Mult}(H^2_d) \text{ and } \psi(b(x))=\phi(x),\text{ }\forall x\in X\}.\label{3} \end{equation}
\par
For the proof of Theorem \ref{1}, we will be making use of the \textit{Riesz-Dunford functional calculus} in the setting of the annulus $A_r.$ Instead of employing the standard Cauchy integral formula, we will adopt the equivalent Laurent series definition. Let $T\in\mathcal{B}(H)$ and suppose that the spectrum $\sigma(T)$ of $T$ is contained in $A_r.$ If $f=\sum_{n\in\Z}a_n z^n$ is any function holomorphic on $A_r,$ then $f(T)\in\mathcal{B}(H)$ is defined as
$$f(T)=\sum_{n\in\Z}a_n T^n.$$
Observe that since $\sigma(T)\subseteq A_r,$ the convergence of the above Laurent series is guaranteed. See e.g. \cite{Conway} for the basic properties of the Riesz-Dunford functional calculus. \\
We now set up the \textit{hereditary functional calculus} on $A_r.$ We say that $h$ is a hereditary function on $A_r$ if $h$ is a mapping from $A_r\times A_r$ to $\C$ and has the property that $$(\lambda,\mu)\mapsto h(\lambda,\bar{\mu})\in\C$$
is a holomorphic function on $A_r\times A_r.$ The set $\text{Her}(A_r)$ of hereditary functions on $A_r$ forms a complete metrizable locally convex topological vector space when equipped with the topology of uniform convergence on compact subsets of $A_r\times A_r.$ \\
If $T\in\mathcal{B}(H)$ with $\sigma(T)\subseteq A_r$ and $h$ is a hereditary function on $A_r,$ then we may define $h(T)\in\mathcal{B}(H)$ by the following procedure. Expand $h$ into a double Laurent series 
$$h(\lambda,\mu)=\sum_{m,n\in\Z}c_{mn}\lambda^m\bar{\mu}^n \hspace{0.3 cm}\text{ for all }\lambda,\mu\in A_r,$$
and then define $h(T)$ by substituting $T$ for $\lambda$ and $T^{*}$ for $\bar{\mu}$: 
$$h(T)=\sum_{m,n\in\Z}c_{mn}T^m(T^*)^n.$$
There is a natural involution $h\mapsto h^*$ on $\text{Her}(A_r),$ defined by
$$h^*(\lambda,\mu)=\overline{h(\mu,\lambda)},\hspace{0.3 cm} \text{ for all      }\lambda,\mu\in A_r.$$
It is easy to see that $h^*(T)=h(T^*).$\\
Finally, we record the following fundamental lemma (the counterpart of Theorem 2.88 in \cite{NewPick} for the annulus) which is essentially the holomorphic version of Moore's theorem on the factorisation of positive semi-definite kernels. It will allow us to decompose positive semi-definite hereditary functions as sums of dyads.  
\begin{lemma} \label{4}
Suppose $U$ is a positive semi-definite hereditary function on $A_r$, then there exists a sequence $\{f_n\}_{n\in\mathbb{N}}$ of functions holomorphic on $A_r$ such that
$$U(\lambda,\mu)=\sum_{n=1}^{\infty}f_n(\lambda)\overline{f_n(\mu)},\hspace{0.2 cm} \forall \lambda,\mu \in A_r,             $$
the series converging uniformly on compact subsets of $A_r\times A_r.$

\end{lemma}
\vspace{0.05 cm}

\small
\section{THE SPACE $\mathscr{H}^2(A_r)$}
\label{spacesection}
\large
Fix $r<1$ and let $A_r=\{r<|z|<1\}$. Denote by $H^2(A_r)$ the classical Hardy space on an annulus. This is the Hilbert function space 
$$H^2(A_r)=\{f\in\text{Hol}(A_r):\sup_{r<\rho<1}\frac{1}{2\pi}\int^{2\pi}_{0}|f(\rho e^{it})|^2\dt<\infty \} $$
equipped with the norm (for $f=\sum_{n\in\Z}a_n z^n$)
$$||f||^2_{H^2(A_r)}=\sum_{-\infty}^{\infty}(r^{2n}+1)|a_n|^2.$$
An important observation is that the multiplier algebra $\text{Mult}(H^2(A_r))$ is isometrically isomorphic to the algebra $H^{\infty}(A_r)$ of bounded holomorphic functions on $A_r.$ \\
Now, we define the space $\mathscr{H}^2(A_r)$ by equipping $H^2(A_r)$ with a different norm: 
$$||f||^2_{\mathscr{H}^2(A_r)}=\sum_{-\infty}^{-1}r^{2n}|a_n|^2+\sum_{0}^{\infty}|a_n|^2.$$
These two norms are equivalent, as 
$$||f||^2_{\mathscr{H}^2(A_r)}\le||f||^2_{H^2(A_r)}=\sum_{-\infty}^{\infty}(r^{2n}+1)|a_n|^2$$ $$\le 2\sum_{-\infty}^{-1}r^{2n}|a_n|^2+2\sum_{0}^{\infty}|a_n|^2=2||f||^2_{\mathscr{H}^2(A_r)} .$$
Hence, \begin{equation}||f||_{\mathscr{H}^2(A_r)}\le||f||_{H^2(A_r)}\le\sqrt{2}||f||_{\mathscr{H}^2(A_r)} ,\label{5}\end{equation}
for every $f\in \mathscr{H}^2(A_r).$ Notice also that the set 
$$\bigg\{\frac{z^n}{r^n} \bigg\}_{n\le -1}\cup\big\{z^n\big\}_{n\ge 0} $$
is an orthonormal basis for $\mathscr{H}^2(A_r).$ Applying Parseval's identity, we can then calculate the kernel function for $\mathscr{H}^2(A_r)$ as follows 
$$k_{A_r}(\lambda,\mu)=\langle k_{A_r}(\cdot,\mu),k_{A_r}(\cdot,\lambda)\rangle$$ $$
=\sum_{-\infty}^{-1}\langle  k_{A_r}(\cdot,\mu), z^n/r^n \rangle\langle  z^n/r^n,k_{A_r}(\cdot,\lambda)\rangle +\sum_0^{\infty}\langle k_{A_r}(\cdot,\mu),z^n\rangle\langle z^n,k_{A_r}(\cdot,\lambda)\rangle $$ $$=\sum_{-\infty}^{-1}\frac{\lambda^n}{r^n}{}
 \frac{\bar{\mu}^n}{r^n}  +\sum_{0}^{\infty}\lambda^n\bar{\mu}^n           $$
$$=(1-r^2)\frac{1}{\big(1-\frac{r^2}{\lambda\bar{\mu}}\big)\big(1-\lambda\bar{\mu}\big)}, \hspace{0.2 cm} \forall \lambda,\mu\in A_r.          $$
There are a few interesting observations we can make here. Recall that $a_2(\lambda,\mu)$ (where $\lambda=(\lambda_1,\lambda_2)$ and $\mu=(\mu_1,\mu_2)$) denotes the reproducing kernel of the Drury-Arveson space $H^2_2$ defined on the $2$-dimensional complex unit ball $\mathbb{B}_2=\{(z_1,z_2):|z_1|^2+|z_2|^2<1\}.$ Denote also by $s_2(\lambda,\mu)$ the kernel of the classical Hardy space $H^2(\DD^2)$ defined on the bidisk $\DD^2=\{(z_1,z_2):|z_1|,|z_2|<1\}$, 
            $$s_2(\lambda,\mu)=\cfrac{1}{(1-\lambda_1\overline{\mu_1})(1-\lambda_2\overline{\mu_2})}.$$
A short calculation then leads us to the equalities 
$$  k_{A_r}(\lambda,\mu)=$$
\begin{equation}
    \label{6}
   =\bigg(\frac{1-r^2}{1+r^2}\bigg)a_2\bigg(\bigg(\frac{\lambda}{\sqrt{r^2+1}},\frac{r}{\sqrt{r^2+1}}\frac{1}{\lambda}\bigg)  ,\bigg(\frac{\mu}{\sqrt{r^2+1}},\frac{r}{\sqrt{r^2+1}}\frac{1}{\mu}\bigg) \bigg)
\end{equation}
\begin{equation}
    \label{7}
    =(1-r^2)s_2\bigg(\bigg(\lambda,\frac{r}{\lambda}\bigg),\bigg(\mu,\frac{r}{\mu} \bigg) \bigg),
\end{equation}
for every $\lambda$ and $\mu$ in $A_r.$ We can now apply the pull-back theorem for reproducing kernels (see e.g. Theorem 5.7 in \cite{IntrotoRep}) to obtain two new descriptions of the norm of $\mathscr{H}^2(A_r)$. By ($\ref{6}$), we obtain that for every $f\in\mathscr{H}^2(A_r):$ 
$$||f||_{\mathscr{H}^2(A_r)}=$$  $$\resizebox{0.99\hsize}{!}{$=\sqrt{\frac{1+r^2}{1-r^2}}\min\{||g||_{H^2_2}:g\in H^2_2 \text{ and } g\Big(\frac{z}{\sqrt{1+r^2}},\frac{r}{\sqrt{1+r^2}}\frac{1}{z}\Big)=f(z), \forall z\in A_r\}$},$$
while (\ref{7}) gives us: 
$$||f||_{\mathscr{H}^2(A_r)}=$$  $$\resizebox{0.99\hsize}{!}{$=\sqrt{\frac{1}{1-r^2}}\min\{||g||_{H^2(\DD^2)}:g\in H^2(\DD^2) \text{ and } g\big(z,\frac{r}{z}\big)=f(z), \forall z\in A_r\}$}.$$
We will now use (\ref{5}) to compare the norm of the multipliers of $\mathscr{H}^2(A_r)$ with the supremum norm on $H^{\infty}(A_r).$
\begin{proposition}\label{8}
 For every $\phi\in H^{\infty}(A_r),$
 $$||\phi||_{\infty}\le ||\phi||_{\textnormal{Mult}(\mathscr{H}^2(A_r))}\le \sqrt{2}||\phi||_{\infty}. $$
 Moreover, the constant $\sqrt{2}$ is the best possible.
 
\end{proposition}
\begin{proof} Since $\mathscr{H}^2(A_r)$ is a reproducing kernel Hilbert space, the inequality $||\phi||_{\infty}\le ||\phi||_{\text{Mult}(\mathscr{H}^2(A_r))}$ is automatic, for all $\phi$ in $\text{Mult}(\mathscr{H}^2(A_r)).$\\ Now, fix $\phi\in H^{\infty}(A_r).$  (\ref{5}) allows us to write
$$||\phi f||_{\mathscr{H}^2(A_r)}\le ||\phi f||_{H^2(A_r)}\le ||\phi||_{\infty}||f||_{H^2(A_r)}\le \sqrt{2}||\phi||_{\infty}||f||_{\mathscr{H}^2(A_r)}, 
$$
for every $f\in\mathscr{H}^2(A_r)$.
Hence, $||\phi||_{\text{Mult}(\mathscr{H}^2(A_r))}\le \sqrt{2}||\phi||_{\infty},$ as desired. \\
Now, to prove that the constant $\sqrt{2}$ is the best possible, we define  $$g_n(z)=\frac{r^n}{z^n}+z^n,\hspace{0.2 cm} \forall z\in A_r,\hspace{0.2 cm}\forall n\ge 1.$$
Then, for every $z\in A_r:$
$$|g_n(z)|\le \frac{r^n}{|z|^n}+|z|^n\le 1+r^n.$$
Hence, $||g_n||_{\infty}=1+r^n$ for all  $n\ge 1.$ Notice also that
$$\frac{||g_n||_{\text{Mult}(\mathscr{H}^2(A_r))}}{||g_n||_{\infty}}=
\frac{||g_n||_{\text{Mult}(\mathscr{H}^2(A_r))}}{1+r^n}
$$
 $$\ge \frac{||g_n\cdot 1||_{\mathscr{H}^2(A_r)}}{1+r^n}  $$
$$=\frac{\sqrt{2}}{1+r^n}\xrightarrow{n\to\infty}\sqrt{2}. $$
This concludes our proof.
\end{proof}\noindent We return to equality (\ref{6}). This can be written equivalently as
$$k_{A_r}(\lambda,\mu)=\frac{1-r^2}{1+r^2}\frac{1}{1-\langle b_r(\lambda),b_r(\mu)\rangle_{\C^2}},\hspace{0.3 cm} \forall \lambda,\mu\in A_r,$$
where $$b_r(\lambda):=\bigg(\frac{\lambda}{\sqrt{r^2+1}},\frac{r}{\sqrt{r^2+1}}\frac{1}{\lambda}\bigg)$$
and $|b_r(\lambda)|<1$ in $A_r.$ \\ Hence, $k_{A_r}$ is a complete Pick kernel. This allows us to draw an interesting connection between the supremum norm of $H^{\infty}(A_r)$ and the multiplier norm of $\text{Mult}(H^2_2),$ formulated as an extension result of holomorphic functions off a subvariety of $\mathbb{B}_2.$
It is the content of Theorem \ref{9}, which we now prove.
\begin{proof}[Proof of Theorem \ref{9}]
Since $\frac{1+r^2}{1-r^2}k_{A_r}$ is a (normalized) complete Pick kernel, we can use (\ref{3}) to deduce that for every $\phi\in\text{Mult}(\mathscr{H}^2(A_r))=H^{\infty}(A_r),$
$$||\phi||_{\textnormal{Mult}(\mathscr{H}^2(A_r))}=$$ $$=min\{||\psi||_{\textnormal{Mult}(H^2_2)}:\psi\in\text{Mult}(H^2_2) \text{ and } \psi(b_r(z))=\phi(z),\text{ }\forall z\in A_r\}.$$
Applying Proposition \ref{8} then concludes the proof.
\end{proof}
\vspace{0.05 cm}
\begin{remark}
A rescaled version of the space $\mathscr{H}^2(A_r)$ was also considered by Arcozzi, Rochberg, Sawyer in \cite{Rochberg}. There, the authors proved the much more general fact that every Hardy space over a finitely connected domain with smooth boundary curves admits an equivalent norm originating from a complete Pick kernel. 

\end{remark}
\small
\section{PROOF OF THEOREM \ref{1}} \label{proofsection}
\large
We will now prove Theorem \ref{1}, which shows that the class of operators $\mathcal{F}_r=\{T\in\mathcal{B}(H): r^2T^{-1}(T^{-1})^*+TT^*\le r^2+1,\hspace{0.08 cm}\sigma(T)\subset A_r\}$ can be naturally associated with the space $\mathscr{H}^2(A_r)$.
\large
\begin{proof}[Proof of Theorem \ref{1}]
Since we have already established Proposition \ref{8}, it remains to show that 
$$\sup_{T\in\mathcal{F}_r}||\phi(T)||=||\phi||_{\textnormal{Mult}(\mathscr{H}^2(A_r))}, $$
for every $\phi\in H^{\infty}(A_r)$.  \\
First, suppose $||\phi||_{\textnormal{Mult}(\mathscr{H}^2(A_r))}\le 1$. By (\ref{2}), we obtain the existence of a positive-semidefinite kernel $U:A_r\times A_r\to\C$ such that
$$(1-\phi(\lambda)\overline{\phi(\mu)})k_{A_r}(\lambda,\mu)=U(\lambda,\mu),$$
hence we have the model formula
$$1-\phi(\lambda)\overline{\phi(\mu)}=\frac{U(\lambda,\mu)}{1-r^2}(1+r^2-r^2/\lambda\bar{\mu}-\lambda\bar{\mu}).$$
Evidently, $U$ is a positive semi-definite element of $\text{Her}(A_r)$. Applying Lemma \ref{4}, we obtain  the existence of a sequence $\{f_n\}_{n\in\mathbb{N}}$ of elements of $\text{Hol}(A_r)$ such that
$$ 1-\phi(\lambda)\overline{\phi(\mu)}=\frac{1}{1-r^2}\sum_{n=1}^{\infty}f_n(\lambda)(1+r^2-r^2/\lambda\bar{\mu}-\lambda\bar{\mu})\overline{f_n(\mu)},     $$
with uniform convergence on compact subsets of $A_r\times A_r$.
\\ Now, let $T\in \mathcal{F}_r.$ We can view each side of our previous equality as a hereditary function on $A_r$ and substitute $T$ into both sides using our hereditary functional calculus on $A_r$ (since $\sigma(T)\subseteq A_r$). This results in the equality 
\begin{equation} \label{11}
    1-\phi(T)\phi(T)^*=\frac{1}{1-r^2}\sum_{n=1}^{\infty}f_n(T)(1+r^2-r^2T^{-1}(T^{-1})^*-TT^*)f_n(T)^*.
\end{equation}
However, observe that since $T\in\mathcal{F}_r$ we can write
$$1+r^2-r^2T^{-1}(T^{-1})^*-TT^*\ge 0$$
$$\RR f_n(T)(1+r^2-r^2T^{-1}(T^{-1})^*-TT^*)f_n(T)^*\ge 0 $$
$$\RR \sum_{n=1}^{\infty}f_n(T)(1+r^2-r^2T^{-1}(T^{-1})^*-TT^*)f_n(T)^*\ge 0.$$
By (\ref{11}), we can then conlude that 
$$  1-\phi(T)\phi(T)^*\ge 0 $$
$$\RR ||\phi(T)||\le 1.$$
We have showed that 
$$\sup_{T\in\mathcal{F}_r}||\phi(T)||\le ||\phi||_{\textnormal{Mult}(\mathscr{H}^2(A_r))}.$$
We now show the reverse inequality. First, we prove a lemma.
\begin{lemma}\label{13}
Every $T\in \mathcal{B}(H)$ such that $r^2T^{-1}(T^{-1})^*+TT^*\le r^2+1$ satisfies $$r^2\le TT^*\le 1.$$
Thus, we also have that $\sigma(T)\subset \overline{A_r}$ for every such operator.
\end{lemma}
\begin{proof}[Proof of Lemma \ref{13}]
Suppose instead that $||T^*||=||T||=1+\delta>1.$ Thus, for any $\epsilon\in(0,\delta)$ there exists $y\in H$ with $||y||=1$ such that $||T^*y||>1+\delta-\epsilon.$  Since $||T^*x||\le(1+\delta)||x||,$ we also obtain $\frac{1}{1+\delta}||x||\le ||(T^{-1})^*x||,$ for every $x\in H.$ We can now write
$$\frac{r^2}{(1+\delta)^2}+(1+\delta-\epsilon)^2<r^2||(T^{-1})^*y||^2+||T^*y||^2$$ $$=\langle (r^2T^{-1}(T^{-1})^*+TT^*)y,y\rangle\le r^2+1.$$
Letting $\epsilon\to 0,$ we obtain
$$\frac{r^2}{(1+\delta)^2}+(1+\delta)^2\le r^2+1,$$
a contradiction. Hence, $||T||\le 1$ and an analogous argument shows that $||rT^{-1}||\le 1$ as well.
\end{proof}
\noindent
Our next step will be to prove the following special case.
\begin{lemma} \label{14}
Let $\phi$ be holomorphic in a neighborhood of $\overline{A_r}$ with the property that $||\phi(T)||\le 1$ for all $T\in \mathcal{B}(H)$ such that $r^2T^{-1}(T^{-1})^*+TT^*\le r^2+1$. Then, $||\phi||_{\textnormal{Mult}(\mathscr{H}^2(A_r))}\le 1. $
\end{lemma}
\begin{proof}[Proof of Lemma \ref{14}]
If $T\in \mathcal{B}(H)$ is such that $r^2T^{-1}(T^{-1})^*+TT^*\le r^2+1$, then by Lemma \ref{13} $T$ also satisfies $r^2\le TT^*\le 1.$ In particular, $\sigma(T)$ has to lie in $\overline{A_r}$ and so the operator $\phi(T)$ is indeed well-defined whenever $\phi\in\text{Hol}(\overline{A_r}).$ \\
Now, suppose that $\phi$ satisfies the given hypotheses and consider the bilateral shift operator $(Sf)(z)=zf(z)$ defined on $\mathscr{H}^2(A_r).$ A standard computation shows that 
$S^*k_{A_r}(\cdot,\lambda)=\bar{\lambda}k_{A_r}(\cdot,\lambda),$ for every $\lambda\in A_r.$ Notice also that for every $\lambda, \mu$ in $A_r,$
$$\langle (r^2+1-r^2S^{-1}(S^{-1})^*-SS^*)k_{A_r}(\cdot,\mu),k_{A_r}(\cdot,\lambda)\rangle$$
$$=(r^2+1)\langle k_{A_r}(\cdot,\mu),k_{A_r}(\cdot,\lambda)\rangle-r^2\langle (S^{-1})^*k_{A_r}(\cdot,\mu),(S^{-1})^*k_{A_r}(\cdot,\lambda)\rangle$$ $$-\langle S^*k_{A_r}(\cdot,\mu),S^*k_{A_r}(\cdot,\lambda)\rangle $$
$$=(r^2+1)k_{A_r}(\lambda,\mu)-\frac{r^2}{\lambda\bar{\mu}}k_{A_r}(\lambda,\mu)-\lambda\bar{\mu}k_{A_r}(\lambda,\mu)$$
$$=1-r^2,$$
a (trivial) positive semi-definite kernel on $A_r\times A_r.$. Since linear combinations of kernel functions are dense in $\mathscr{H}^2(A_r)$, our previous equality implies that 
$$r^2+1-r^2S^{-1}(S^{-1})^*-SS^*\ge 0.$$
But our hypotheses on $\phi$ then allow us to deduce that
$$||\phi||_{\textnormal{Mult}(\mathscr{H}^2(A_r))}=||\phi(S)||\le\sup_{T\in\mathcal{F}_r}||\phi(T)||\le 1,$$
which concludes the proof of the lemma.
\end{proof}
\noindent
To complete our main proof, we will apply an approximation argument to extend the previous special case to every multiplier of $\mathscr{H}^2(A_r).$ \\ Suppose $\phi\in\text{Hol}(A_r)$ is such that $\sup_{T\in\mathcal{F}_r}||\phi(T)||\le 1.$
For $n>2/(1-r)$, define $A_{r,n}:=\{r+1/n<|\lambda|<1-1/n \}.$ We will be needing the following lemma, the proof of which is just a simple calculation. 
\begin{lemma}\label{12}
The following two inequalities hold for all $n>2/(1-r)$:
$$\frac{\big(r+\frac{1}{n} \big)^2}{1+\bigg(\cfrac{r+\frac{1}{n}}{1-\frac{1}{n}} \bigg)^2}\ge \frac{r^2}{1+r^2}, $$ 

$$\frac{1}{\big(1-\frac{1}{n} \big)^2+\big(r+\frac{1}{n} \big)^2 } \ge \frac{1}{1+r^2}.$$

\end{lemma}
\noindent 
Now, define the classes of operators 
 $$\mathcal{F}_{r,n}=\{T\in \mathcal{B}(H):$$ $$[r+(1/n)]^2T^{-1}(T^{-1})^*+[1/(1-n)]^2TT^*\le 1+[(r+(1/n))/(1-(1/n))]^2 \}$$
 and also the family of kernels
 $$k_{r,n}(\lambda,\mu):=\cfrac{1}{\Bigg(   1-\cfrac{\big(r+1/n  \big)^2}{\bar{\mu}\lambda}\Bigg)\Bigg( 1-\cfrac{\bar{\mu}\lambda}{\big(1-1/n\big)^2}\Bigg)}$$ $$=\cfrac{1}{\Bigg(1+\cfrac{(r+1/n)^2}{(1-1/n)^2}-\cfrac{\big(r+1/n  \big)^2}{\bar{\mu}\lambda}-\cfrac{\bar{\mu}\lambda}{\big(1-1/n\big)^2} \Bigg)}.$$
Each $k_{r,n}$ is a positive-semidefinite kernel on $A_{r,n}\times A_{r,n},$ simply a rescaled version of $k_{A_r}.$ Denote by $\mathscr{H}^2(A_{r,n})$ the corresponding Hilbert space of holomorphic functions on $A_{r,n}.$
 \\ 
 Now, let $T\in\mathcal{F}_{r,n}$. By the appropriately rescaled version of Lemma \ref{13}, we obtain $\sigma(T)\subseteq \overline{A_{r,n}}\subseteq A_r.$  Observe also that 
 $$ [r+(1/n)]^2T^{-1}(T^{-1})^*+[1/(1-n)]^2TT^*\le 1+[(r+(1/n))/(1-(1/n))]^2   $$

$$\Longleftrightarrow 
\frac{\big(r+\frac{1}{n} \big)^2}{1+\bigg(\cfrac{r+\frac{1}{n}}{1-\frac{1}{n}} \bigg)^2}||(T^{-1})^*x||^2 +
\frac{1}{\big(1-\frac{1}{n} \big)^2+\big(r+\frac{1}{n} \big)^2 } ||T^*x||^2\le ||x||^2, 
$$
for every $x\in H.$
Using our inequalities from Lemma \ref{12}, we obtain 
$$\cfrac{r^2}{r^2+1}||(T^{-1})^*x||^2+\cfrac{1}{r^2+1}||T^*x||^2 $$
$$\le    \frac{\big(r+\frac{1}{n} \big)^2}{1+\bigg(\cfrac{r+\frac{1}{n}}{1-\frac{1}{n}} \bigg)^2}||(T^{-1})^*x||^2 +
\frac{1}{\big(1-\frac{1}{n} \big)^2+\big(r+\frac{1}{n} \big)^2 } ||T^*x||^2$$ $$\le ||x||^2, \hspace{0.3 cm} \forall x\in H.                $$
Thus, $r^2T^{-1}(T^{-1})^*+TT^*\le r^2+1$, which means that $T\in\mathcal{F}_r.$ By our assumptions on $\phi,$ we then obtain that $||\phi(T)||\le 1.$ \\
To sum up, we have proved (for every $n>2/(1-r)$) that $\phi$, a function holomorphic on a neighborhood of $\overline{A_{r,n}}\hspace{0.05 cm} ,$ satisfies $||\phi(T)||\le 1$ for all $T\in \mathcal{B}(H)$ such that $$[r+(1/n)]^2(T^{-1})^*T^{-1}+[1/(1-n)]^2T^*T\le 1+[(r+(1/n))/(1-(1/n))]^2.$$
The appropriately rescaled version of Lemma \ref{14} now allows us to conclude that $||\phi||_{\textnormal{Mult}(\mathscr{H}^2(A_{r,n}))}\le 1$ and so there exists a positive semi-definite hereditary function $h_n:A_{r,n}\times A_{r,n}\to\C$ such that 
$$1-\phi(\lambda)\overline{\phi(\mu)}=(1/k_{r,n}(\lambda,\mu))h_n(\lambda,\mu) $$
\begin{equation}\label{30}
  =\Bigg(1+\cfrac{(r+1/n)^2}{(1-1/n)^2}-\cfrac{\big(r+1/n  \big)^2}{\bar{\mu}\lambda}-\cfrac{\bar{\mu}\lambda}{\big(1-1/n\big)^2}\Bigg) h_n(\lambda,\mu),
\end{equation}
for all $\lambda,\mu \in A_{r,n}$ and for every $n>2/(1-r).$\\
Let $K\subseteq A_r\times A_r$ be compact and fix $N>2/(1-r)$ large enough so that $K\subseteq A_{r,n}\times A_{r,n}$ for every $n\ge N$. Then, for every such $n$ and for every $\lambda\in K$ we have 
$$|h_n(\lambda,\lambda)|=h_n(\lambda,\lambda)=(1-|\phi(\lambda)|^2)k_{r,n}(\lambda,\lambda)$$ $$\le \sup_{z\in K} \Bigg[(1-|\phi(z)|^2)\cfrac{1}{\big(   1-(r+1/n)^2/|z|^2\big)\big( 1-|z|^2/\big(1-1/n)^2\big)}\Bigg]$$
$$\le \sup_{z\in K} \Bigg[(1-|\phi(z)|^2)\cfrac{1}{\big(   1-(r+1/N)^2/|z|^2\big)\big( 1-|z|^2/\big(1-1/N)^2\big)}\Bigg]$$ \begin{equation}\label{40} =M<\infty,\end{equation}
where $M$ is independent of $n.$ Notice now that by Lemma \ref{4} there exists (for every $n\in\N$) a function $u_n:A_{r,n}\to l^2$ with the property that 
$$h_n(\lambda,\mu)=\langle u_n(\lambda),u_n(\mu)\rangle_{l^2}, \hspace{0.3 cm} \text{in } A_{r,n}\times A_{r,n}.$$
Hence, using the Cauchy-Schwarz inequality and the bound (\ref{40}) we can write 
$$|h_n(\lambda,\mu)|^2\le |h_n(\lambda,\lambda)||h_n(\mu,\mu)|\le M^2,$$
for every $\lambda,\mu\in K$ and for every $n\ge N$. In other words, the sequence of holomorphic functions $\{(\lambda,\mu)\mapsto h_n(\lambda,\bar{\mu})\}_{n\ge N}$ is uniformly bounded on $K.$
By Montel's theorem and the completeness of $\text{Her}(A_r)$, we can then deduce the existence of an element $h\in\text{Her}(A_r)$ with the property that $h_{n_k}\to h$
uniformly on compact subsets of $A_r\times A_r$ for some subsequence $\{n_k\}.$ Since every $h_{n_k}$ is positive semi-definite, the same must be true for $h$ as well. Now, equality (\ref{30}) combined with the convergence $h_{n_k}\to h$ gives
$$1-\phi(\lambda)\overline{\phi(\mu)}=(1+r^2-r^2/\lambda\bar{\mu}-\lambda\bar{\mu})h(\lambda,\mu) \hspace{0.3 cm} \text{on } A_r\times A_r,$$
and so $$(1-\phi(\lambda)\overline{\phi(\mu)})k_{A_r}(\lambda,\mu)\ge 0 \hspace{0.3 cm} \text{on } A_r\times A_r.$$
Thus, $||\phi||_{\textnormal{Mult}(\mathscr{H}^2(A_r))}\le 1$ and our proof is complete.
\end{proof}
\begin{remark}
The class $\mathcal{F}_r$ was also considered in the recent preprint \cite{bello2021operator} of Bello, Yakubovich, where the authors obtained, with an alternative approach, that $A_r$ is a complete $\sqrt{2}$-spectral set for every $T\in\mathcal{F}_r$.

\end{remark}
\begin{remark} By Lemma \ref{13}, every $T\in \mathcal{B}(H)$ such that $r^2T^{-1}(T^{-1})^*+TT^*\le r^2+1$ also satisfies $r^2\le TT^*\le 1$. The converse assertion is not true, even if we restrict ourselves to $2\times 2$ matrices. Indeed, define $A\in\mathcal{B}(\C^2)$ by 
$$A=\begin{bmatrix} 
\sqrt{r} & 1-r \\
0 & \sqrt{r} 
\end{bmatrix}.$$
A short computations shows that $||A||=||rA^{-1}||=1.$ However, notice that 
$$\Big\langle \Big(r^2+1-r^2A^{-1}(A^{-1})^*-AA^*\Big)\begin{bmatrix}1 \\\sqrt{r} \end{bmatrix},\begin{bmatrix}1 \\\sqrt{r} \end{bmatrix}\Big\rangle$$
$$=r^2(r+1-1/r-(2-1/r)^2),$$
which is negative for all $r\in(0,1).$

\end{remark}

 \printbibliography

@book {Pick,
    AUTHOR = {Agler, Jim and McCarthy, John E.},
     TITLE = {Pick interpolation and {H}ilbert function spaces},
    SERIES = {Graduate Studies in Mathematics},
    VOLUME = {44},
 PUBLISHER = {American Mathematical Society, Providence, RI},
      YEAR = {2002},
      ISBN = {0-8218-2898-3} }

@article {CompleteNP,
    AUTHOR = {Agler, Jim and McCarthy, John E.},
     TITLE = {Complete {N}evanlinna-{P}ick kernels},
   JOURNAL = {J. Funct. Anal.},
  FJOURNAL = {Journal of Functional Analysis},
    VOLUME = {175},
      YEAR = {2000},
    NUMBER = {1},
     PAGES = {111--124}
}

@book{NewPick,
author={Agler, Jim and McCarthy, John E. and Young, Nicholas J.},
place={Cambridge}, series={Cambridge Tracts in Mathematics}, title={Operator Analysis: Hilbert Space Methods in Complex Analysis}, publisher={Cambridge University Press}, year={2020}, collection={Cambridge Tracts in Mathematics}}

@book {Conway,
    AUTHOR = {Conway, John B.},
     TITLE = {A course in functional analysis},
    SERIES = {Graduate Texts in Mathematics},
    VOLUME = {96},
   EDITION = {Second},
 PUBLISHER = {Springer-Verlag, New York},
      YEAR = {1990},
   MRCLASS = {46-01 (47-01)},
  MRNUMBER = {1070713},
}

@article {Ando,
    AUTHOR = {And\^{o}, T.},
     TITLE = {On a pair of commutative contractions},
   JOURNAL = {Acta Sci. Math. (Szeged)},
  FJOURNAL = {Acta Universitatis Szegediensis. Acta Scientiarum
              Mathematicarum},
    VOLUME = {24},
      YEAR = {1963},
     PAGES = {88--90}
}

@article {Drury,
    AUTHOR = {Drury, S. W.},
     TITLE = {A generalization of von {N}eumann's inequality to the complex
              ball},
   JOURNAL = {Proc. Amer. Math. Soc.},
  FJOURNAL = {Proceedings of the American Mathematical Society},
    VOLUME = {68},
      YEAR = {1978},
    NUMBER = {3},
     PAGES = {300--304}
}

@article {vNeumann,
    AUTHOR = {von Neumann, Johann},
     TITLE = {Eine {S}pektraltheorie f\"{u}r allgemeine {O}peratoren eines
              unit\"{a}ren {R}aumes},
   JOURNAL = {Math. Nachr.},
  FJOURNAL = {Mathematische Nachrichten},
    VOLUME = {4},
      YEAR = {1951},
     PAGES = {258--281}
}

@article {Varopoulos,
    AUTHOR = {Varopoulos, Nicholas Th.},
     TITLE = {Sur une in\'{e}galit\'{e} de von {N}eumann},
   JOURNAL = {C. R. Acad. Sci. Paris S\'{e}r. A-B},
  FJOURNAL = {Comptes Rendus Hebdomadaires des S\'{e}ances de l'Acad\'{e}mie des
              Sciences. S\'{e}ries A et B},
    VOLUME = {277},
      YEAR = {1973},
     PAGES = {A19--A22}
}

@book {IntrotoRep,
    AUTHOR = {Paulsen, Vern I. and Raghupathi, Mrinal},
     TITLE = {An introduction to the theory of reproducing kernel {H}ilbert
              spaces},
    SERIES = {Cambridge Studies in Advanced Mathematics},
    VOLUME = {152},
 PUBLISHER = {Cambridge University Press, Cambridge},
      YEAR = {2016},
     PAGES = {x+182}
}

@article {Rochberg,
    AUTHOR = {Arcozzi, N. and Rochberg, R. and Sawyer, E.},
     TITLE = {Carleson measures for the {D}rury-{A}rveson {H}ardy space and
              other {B}esov-{S}obolev spaces on complex balls},
   JOURNAL = {Adv. Math.},
  FJOURNAL = {Advances in Mathematics},
    VOLUME = {218},
      YEAR = {2008},
    NUMBER = {4},
     PAGES = {1107--1180},
      ISSN = {0001-8708},
}

@article{bello2021operator,
      title={An operator model in the annulus}, 
      author={Glenier Bello and Dmitry Yakubovich},
      year={2021},
      eprint={2106.08757},
      archivePrefix={arXiv},
      primaryClass={math.FA}
}

\end{document}